\documentclass{article}%
\usepackage{amsmath}
\usepackage{amsfonts}
\usepackage{amssymb}
\usepackage{graphicx}%
\newtheorem{theorem}{Theorem}
\newtheorem{acknowledgement}[theorem]{Acknowledgement}

\newtheorem{claim}[theorem]{Claim}

\newtheorem{definition}[theorem]{Definition}

\newtheorem{lemma}[theorem]{Lemma}

\newtheorem{problem}[theorem]{Problem}
\newtheorem{proposition}[theorem]{Proposition}

\newenvironment{proof}[1][Proof]{\noindent\textbf{#1.} }{\ \rule{0.5em}{0.5em}}
\begin{document}

\date{}
\title{On $\left(  1,\omega_{1}\right)  $\emph{-}weakly universal functions}
\author{Osvaldo Guzm\'{a}n\thanks{The author was supported by NSERC of Canada.\newline%
\textit{keywords: }Universal graphs, weakly universal graphs, Sacks model,
Cohen model.\newline\textit{AMS Classification: }03E35, 03E05 , 03E40}}
\maketitle

\begin{abstract}
A function $U:\left[  \omega_{1}\right]  ^{2}\longrightarrow\omega$ is called
$\left(  1,\omega_{1}\right)  $\emph{-weakly universal }if for every function
$F:\left[  \omega_{1}\right]  ^{2}\longrightarrow\omega$ there is an injective
function $h:\omega_{1}\longrightarrow\omega_{1}$ and a function $e:\omega
\longrightarrow\omega$ such that $F\left(  \alpha,\beta\right)  =e\left(
U\left(  h\left(  \alpha\right)  ,h\left(  \beta\right)  \right)  \right)  $
for every $\alpha,\beta\in\omega_{1}$. We will prove that it is consistent
that there are no $\left(  1,\omega_{1}\right)  $\emph{-}weakly universal
functions, this answers a question of Shelah and Stepr\={a}ns. In fact, we
will prove that there are no $\left(  1,\omega_{1}\right)  $\emph{-}weakly
universal functions in the Cohen model and after adding $\omega_{2}$ Sacks
reals side-by-side. However, we show that there are $\left(  1,\omega
_{1}\right)  $\emph{-}weakly universal functions in the Sacks model. In
particular, the existence of such graphs is consistent with $\clubsuit$ and
the negation of the Continuum Hypothesis.

\end{abstract}

\subsection{Introduction and Preliminaries}

A graph $U:\left[  \omega_{1}\right]  ^{2}\longrightarrow2$ is called
\emph{universal }if for every graph $F:\left[  \omega_{1}\right]
^{2}\longrightarrow2$ there is an injective function $h:\omega_{1}%
\longrightarrow\omega_{1}$ such that $F\left(  \alpha,\beta\right)  =U\left(
h\left(  \alpha\right)  ,h\left(  \beta\right)  \right)  $ for each
$\alpha,\beta\in\omega_{1}.$ It is easy to see that universal graphs exist
assuming the Continuum Hypothesis, and in \cite{ShelahUniversal1} and
\cite{ShelahUniversal2} Shelah showed that the existence of universal
functions is consistent with the failure of \textsf{CH.} In
\cite{MeklerUniversal} Mekler showed that the existence of universal functions
$U:\left[  \omega_{1}\right]  ^{2}\longrightarrow\omega$ is also consistent
with the failure of the Continuum Hypothesis. Universal graphs and functions
were recently studied by Shelah and Stepr\={a}ns in
\cite{UniversalGraphsandFunctionsonw1}, where they showed that the existence
of universal graphs is consistent with several values of $\mathfrak{b}$ and
$\mathfrak{d}.$ They also considered several variations of universal
functions, in particular, the following notion was studied:

\begin{definition}
A function $U:\left[  \omega_{1}\right]  ^{2}\longrightarrow\omega$ is
$\left(  1,\omega_{1}\right)  $-weakly universal\emph{ }if for every
$F:\left[  \omega_{1}\right]  ^{2}\longrightarrow\omega$ there is an injective
function $h:\omega_{1}\longrightarrow\omega_{1}$ and a function $e:\omega
\longrightarrow\omega$ such that $F\left(  \alpha,\beta\right)  =e\left(
U\left(  h\left(  \alpha\right)  ,h\left(  \beta\right)  \right)  \right)  $
for every $\alpha,\beta\in\omega_{1}$.
\end{definition}

\qquad\ \ 

Evidently, every universal function is $\left(  1,\omega_{1}\right)  $-weakly
universal. In \cite{UniversalGraphsandFunctionsonw1} it was proved that a
function $U:\left[  \omega_{1}\right]  ^{2}\longrightarrow\omega$ is $\left(
1,\omega_{1}\right)  $-weakly universal\emph{ }if and only if for every
$F:\left[  \omega_{1}\right]  ^{2}\longrightarrow\omega$ there is an injective
function $h:\omega_{1}\longrightarrow\omega_{1}$ such that if $F\left(
\alpha,\beta\right)  \neq F\left(  \alpha_{1},\beta_{1}\right)  $ then
$U\left(  h\left(  \alpha\right)  ,h\left(  \beta\right)  \right)  \neq
U\left(  h\left(  \alpha_{1}\right)  ,h\left(  \beta_{1}\right)  \right)  $
for every $\alpha,\beta,\,\alpha_{1},\beta_{1}\in\omega_{1}.$

\qquad\ \ \ \ 

In an unpublished note of Tanmay Inamdar, it was proved that $\left(
1,\omega_{1}\right)  $-weakly universal functions exist assuming Martin's
axiom for Knaster forcings (see \cite{UniversalGraphsandFunctionsonw1}). In
\cite{UniversalGraphsandFunctionsonw1} Shelah and Stepr\={a}ns asked the following:

\begin{problem}
[\cite{UniversalGraphsandFunctionsonw1}]Is there (in \textsf{ZFC}) a $\left(
1,\omega_{1}\right)  $-weakly universal function?
\end{problem}

\qquad\ \ \ \ 

In this note, we answer the previous question in the negative. For more on
universal graphs and functions, the reader may consult
\cite{UniversalFunctions} and \cite{UniversalGraphsandFunctionsonw1}.

\qquad\ \ \ \qquad\ \ \qquad\ \ \ \ \ \ 

Recall that $\clubsuit$ is the following statement:

\begin{description}
\item[$\clubsuit)$] There is a family $\left\{  S_{\alpha}\mid\alpha\in
LIM\left(  \omega_{1}\right)  \right\}  $ such that each $S_{\alpha}$ is an
unbounded subset of $\alpha$ and for every $X\in\left[  \omega_{1}\right]
^{\omega_{1}}$ the set $\left\{  \alpha\mid S_{\alpha}\subseteq X\right\}  $
is stationary.
\end{description}

\qquad\ \ \ \qquad\ \ 

The principle $\clubsuit$ is a weakening of the $\Diamond$ principle. It is
well known that $\clubsuit$ is consistent with the failure of the Continuum
Hypothesis (see \cite{ProperandImproper}, \cite{SticksandClubs},
\cite{LifeintheSacksModel} or
\cite{CardinalInvariantsoftheContinuumandCombinatoricsonUncountableCardinals}%
). The \emph{stick principle }(introduced in \cite{Stick})\emph{ }is a
weakening of $\clubsuit:$

\begin{description}
\item[$_{\mid}^{\bullet})$] There is a family $\left\{  S_{\alpha}\mid
\alpha\in\omega_{1}\right\}  \subseteq\left[  \omega_{1}\right]  ^{\omega}$
such that for every $X\in\left[  \omega_{1}\right]  ^{\omega_{1}}$ there is an
$\alpha\in\omega_{1}$ such that $S_{\alpha}\subseteq X.$
\end{description}

\qquad\ \ 

It is easy to see that the stick principle is a consequence of both
$\clubsuit$ and \textsf{CH}. For more on $\clubsuit$ and $_{\mid}^{\bullet}$
the reader may consult
\cite{CardinalInvariantsoftheContinuumandCombinatoricsonUncountableCardinals},
\cite{SticksandClubs}, \cite{Diamantesubd} and \cite{SimilarbutnottheSame}.

\qquad\ \ 

We say that a tree $p\subseteq2^{<\omega}$ is a \emph{Sacks tree }if for every
$s\in p$ there is $t\in p$ extending $s$ such that $t^{\frown}0,$ $t^{\frown
}1\subseteq p.$ The set of all Sacks trees is denoted by $\mathbb{S}$ and we
order it by extension. Given an ordinal $\alpha,$ by $\mathbb{S}^{\alpha}$ we
will denote the countable support product of $\alpha$ copies of Sacks forcing
and by $\mathbb{S}_{\alpha}$ we denote $\alpha$-iteration of $\mathbb{S}$ with
countable support. By the \emph{Sacks model} we mean the model obtained after
forcing with $\mathbb{S}_{\omega_{2}}$ and by the \emph{side-by-side Sacks
model} we mean the model obtained after forcing with $\mathbb{S}^{\omega_{2}}$
to a model of G\textsf{CH}$\mathsf{.}$ Although the partial orders
$\mathbb{S}^{\omega_{2}}$ and $\mathbb{S}_{\omega_{2}}$ are not forcing
equivalent, they share very similar features. It is then interesting to point
out differences between this two forcing notions. Some of the main differences
between them are the following:

\begin{enumerate}
\item In the Sacks model every subset of reals of size $\omega_{2}$ can be
mapped continuously onto the reals, while in the side-by-side Sacks model this
is not the case (see \cite{SacksMiller}).

\item The cardinal invariant $\mathfrak{hm}$\footnote{The cardinal invariant
$\mathfrak{hm}$ is the smallest size of a family of $c_{min}$-monochromatic
sets required to cover the Cantor space (where $c_{min}\left(  x,y\right)  $
is the parity of the largest initial segment common to both $x$ and $y$). It
is known that $\mathfrak{c}^{-},cof\left(  \mathcal{N}\right)  \leq
\mathfrak{hm}$ (see \cite{Convexdecompositionsintheplane}). It is an open
question of Geschke if the inequality $\mathfrak{hm<r}$ is consistent. In a
yet unpublished work, the author proved that the inequality $\mathfrak{hm<u}$
is consistent.} is evaluated differently on the Sacks model and in the
side-by-side Sacks model (see \cite{Convexdecompositionsintheplane}).

\item The \textsf{CPA }axioms hold in the Sacks model but not in the
side-by-side Sacks model (see \cite{CPAbook}).
\end{enumerate}

\qquad\ \ \ \qquad\ \ 

In this note, we will point out another difference: There are $\left(
1,\omega_{1}\right)  $-weakly universal functions in the Sacks model, while
there are no such graphs in the side-by-side Sacks model. In
\cite{UniversalFunctions} it was proved that $_{\mid}^{\bullet}+\mathfrak{c}%
>\omega_{1}$ implies that there is no universal function $U:\left[  \omega
_{1}\right]  ^{2}\longrightarrow2.$ However, our results show that the
existence of $\left(  1,\omega_{1}\right)  $-weakly universal functions is
even consistent with $\clubsuit+\mathfrak{c}>\omega_{1}.$

\subsection{The countable support product of Sacks forcing}

The \emph{Sacks side-by-side model} is the model obtained by forcing with
$\mathbb{S}^{\omega_{2}}$ over a model of the Generalized Continuum
Hypothesis. We will prove that there are no $\left(  1,\omega_{1}\right)
$-weakly universal graphs in the Sacks side-by-side model.

\qquad\ \ 

We will need the following lemma:\qquad\ \ 

\begin{lemma}
There is a function $\pi:2^{\omega}\longrightarrow\omega^{\omega}$ such that
for every $r\in2^{\omega},$ if $f$ is an infinite partial function such that
$f\subseteq\pi\left(  r\right)  ,$ then $r$ is definable from $f.$
\end{lemma}

\begin{proof}
Let $h:\omega\longrightarrow2^{<\omega}$ be a definable bijection. We define
$\pi:2^{\omega}\longrightarrow\omega^{\omega}$ as follows: if $r\in2^{\omega}$
and $n\in\omega$ then $\pi\left(  r\right)  \left(  n\right)  =m$ if $m$ is
the least natural number such that $h\left(  m\right)  $ is an initial segment
of $r$ and $h\left(  m\right)  $ has length at least $n.$ It is easy to see
that $\pi$ has the desired property.
\end{proof}

\qquad\ \ \qquad\ \ \ 

Note that if $M$ is a transitive model of \textsf{ZFC} and $r\notin M$ then
$\pi\left(  r\right)  $ does not contain infinite partial functions from $M.$
We will use the following unpublished result of Baumgartner (the reader may
consult \cite{LifeintheSacksModel} for a proof):

\begin{proposition}
[Baumgartner]The principle $\clubsuit$ holds in the Sacks side-by-side model.
\end{proposition}

In fact, we will only use that every uncountable subset of $\omega_{1}$ in the
Sacks side-by-side model contains a countable ground model set. Given a
function $F:\left[  \omega_{1}\right]  ^{2}\longrightarrow\omega$ and
$U:\left[  \omega_{1}\right]  ^{2}\longrightarrow\omega,$ we say that $\left(
h,e\right)  $ is an $\left(  1,\omega_{1}\right)  $\emph{-weakly universal
embedding from }$F$ \emph{to} $U$ if $h:\omega_{1}\longrightarrow\omega_{1}$
is an injective function, $e:\omega\longrightarrow\omega$ and $F\left(
\alpha,\beta\right)  =e\left(  U\left(  h\left(  \alpha\right)  ,h\left(
\beta\right)  \right)  \right)  $ for every $\alpha,\beta\in\omega_{1}$. We
can now prove the following result, answering the problem of Shelah and
Stepr\={a}ns:\qquad

\begin{proposition}
There are no $\left(  1,\omega_{1}\right)  $-weakly universal graphs in the
Sacks side-by-side model.
\end{proposition}

\begin{proof}
Let $p_{0}\in\mathbb{S}^{\omega_{2}}$ and $\dot{U}$ such that $p_{0}%
\Vdash``\dot{U}:\left[  \omega_{1}\right]  ^{2}\longrightarrow\omega
\textquotedblright.$ Since the product of Sacks forcing has the $\omega_{2}%
$-chain condition, we may find $\omega_{1}\leq\beta<\omega_{2}$ such that
$p_{0}\in\mathbb{S}^{\beta}$ and $\dot{U}$ is a $\mathbb{S}^{\beta}$-name.
Given $\alpha<\omega_{1},$ let $\dot{d}_{\alpha}$ be name for $\pi\left(
\dot{r}_{\beta+\alpha}\right)  $ where $\dot{r}_{\beta+\alpha}$ is the name
for the $\left(  \beta+\alpha\right)  $-generic real. For every infinite
$\alpha<\omega_{1},$ we fix an enumeration $\alpha=\left\{  \alpha_{n}\mid
n\in\omega\right\}  .$

\qquad\ \ 

If $G\subseteq\mathbb{S}^{\omega_{2}}$ is a generic filter, in $V\left[
G\right]  $ we define a function $F:\left[  \omega_{1}\right]  ^{2}%
\longrightarrow\omega$ as follows: given $\omega\leq\alpha<\omega_{1}$ we
define $F\left(  \alpha_{n},\alpha\right)  =d_{\alpha}\left(  n\right)  .$ Let
$\dot{F}$ be a name for $F$ and let $\dot{h}$ be a $\mathbb{S}^{\omega_{2}}%
$-name for an injective function from $\omega_{1}$ to $\omega_{1}$ and
$\dot{e}$ be a $\mathbb{S}^{\omega_{2}}$-name for a function from $\omega$ to
$\omega.$ We will see that we can find an extension $q$ of $p_{0}$ that forces
that $(\dot{h},\dot{e})$ is not a $\left(  1,\omega_{1}\right)  $-embedding of
$\dot{F}$ in $\dot{U}.$

\qquad\ \ \ 

We can first find $p_{1}\leq p_{0}$ and a ground model injective function
$g:S\longrightarrow\omega_{1}$ such that $p_{1}\Vdash``g\subseteq\dot
{h}\textquotedblright$ where $S\in\left[  \omega_{1}\right]  ^{\omega}$ (this
is possible since the stick principle holds in the Sacks side-by-side model,
witnessed by the ground model countable sets). Let $M$ be a countable
elementary submodel such that $p_{1},\dot{U},\beta,\dot{F},g,\dot{h},\dot
{e}\in M.$ Let $q\leq p_{1}$ be a $\left(  M,\mathbb{S}^{\omega_{2}}\right)
$-generic condition. We claim that $q$ forces that $(\dot{h},\dot{e})$ is not
an $\left(  1,\omega_{1}\right)  $-embedding of $\dot{F}$ in $\dot{U}.$ Assume
this is not the case, so there is $q_{1}\leq q$ that forces that $(\dot
{h},\dot{e})$ is an $\left(  1,\omega_{1}\right)  $-embedding of $\dot{F}$ in
$\dot{U}.$

\qquad\ \ 

Let $G\subseteq\mathbb{S}^{\omega_{2}}$ be a generic filter such that
$q_{1}\in G.$ Let $X=\beta\cup\left(  M\cap\omega_{2}\right)  $ and define
$G_{X}$ to be the restriction of $G$ to $\mathbb{S}^{X}.$ Since $q_{1}$ is a
$\left(  M,\mathbb{S}^{\omega_{2}}\right)  $-generic condition, it follows
that $\dot{U}\left[  G\right]  ,\dot{e}\left[  G\right]  \in V\left[
G_{X}\right]  .$ Fix $\delta\in\omega_{1}$ such that $S\subseteq\delta$ and
$\beta+\delta\notin X,$ let $A=\left\{  n\in\omega\mid\delta_{n}\in S\right\}
.$ For every $\alpha\in\omega_{1},$ we define $f_{\alpha}:A\longrightarrow
\omega$ the function given by $f_{\alpha}\left(  n\right)  =\dot{e}\left[
G\right]  \left(  \dot{U}\left[  G\right]  \left(  g\left(  \delta_{n}\right)
,\alpha\right)  \right)  $ and note that $f_{\alpha}\in V\left[  G_{X}\right]
$ for every $\alpha\in\omega_{1}.$ Let $\alpha\in\omega_{1}$ such that
$\dot{h}\left[  G\right]  \left(  \delta\right)  =\alpha.$ Since $(\dot
{h}\left[  G\right]  ,\dot{e}\left[  G\right]  )$ is forced to be an $\left(
1,\omega_{1}\right)  $-embedding, if $n\in A$ then we have the following:

\qquad\ \ \ \ %

\begin{tabular}
[t]{lll}%
$d_{\delta}\left(  n\right)  $ & $=$ & $\dot{F}\left[  G\right]  \left(
\delta_{n},\delta\right)  $\\
& $=$ & $\dot{e}\left[  G\right]  \left(  \dot{U}\left[  G\right]  \left(
\dot{h}\left(  \delta_{n}\right)  ,\dot{h}\left(  \delta\right)  \right)
\right)  $\\
& $=$ & $\dot{e}\left[  G\right]  \left(  \dot{U}\left[  G\right]  \left(
g\left(  \delta_{n}\right)  ,\alpha\right)  \right)  $\\
& $=$ & $f_{\alpha}\left(  n\right)  $%
\end{tabular}

\qquad\ \ \bigskip

Hence $f_{\alpha}\subseteq d_{\delta},$ but this is a contradiction since
$r_{\beta+\delta}\notin V\left[  G_{X}\right]  .$
\end{proof}

\subsection{The Cohen model}

The Cohen model is the model obtained after adding $\omega_{2}$-Cohen reals
with finite support to a model of the Generalized Continuum Hypothesis. We
will show that there are no $\left(  1,\omega_{1}\right)  $-weakly universal
graphs in the Cohen model.

\begin{lemma}
If $U:\left[  \omega_{1}\right]  ^{2}\longrightarrow\omega$ then $U$ is not
$\left(  1,\omega_{1}\right)  $-weakly universal after adding $\omega_{2}$
Cohen reals.
\end{lemma}

\begin{proof}
Let $U:\left[  \omega_{1}\right]  ^{2}\longrightarrow\omega.$ We define the
function $H:\left[  \omega^{\omega}\right]  ^{2}\longrightarrow\omega$ given
by $H\left(  x,y\right)  =\left\vert x\wedge y\right\vert $ (where $x\wedge y$
denotes the largest initial segment which both $x$ and $y$ have in common).
Let $\dot{c}_{\alpha}$ be the name for the $\alpha$-Cohen real. Let $\dot{F}$
be a name of a function from $\left[  \omega_{1}\right]  ^{2}$ to $\omega$
such that $\mathbb{C}_{\omega_{2}}\Vdash``\dot{F}\left(  \alpha,\beta\right)
=H\left(  \dot{c}_{\alpha},\dot{c}_{\beta}\right)  \textquotedblright.$

\qquad\ 

Let $p\in\mathbb{C}_{\omega_{2}},$ $\dot{h}$ a name for an injective function
from $\omega_{1}$ to $\omega_{1}$ and $\dot{e}$ a name for a function from
$\omega$ to $\omega.$ We must find $q\leq p$ and $\alpha,\beta\in\omega_{1}$
such that $q\Vdash``\dot{F}\left(  \alpha,\beta\right)  \neq\dot{e}U(\dot
{h}\left(  \alpha\right)  ,\dot{h}\left(  \beta\right)  )\textquotedblright.$
For every $\alpha<\omega_{1}$ we find $p_{\alpha}\leq p$ and $\delta_{\alpha}$
such that $p_{\alpha}\Vdash``\dot{h}\left(  \alpha\right)  =\delta_{\alpha
}\textquotedblright$ and $\alpha\in dom\left(  p_{\alpha}\right)  .$ By the
usual pruning arguments, we may find $X\in\left[  \omega_{1}\right]
^{\omega_{1}},$ $R\in\left[  \omega_{2}\right]  ^{<\omega},$ $\overline{p}%
\in\mathbb{C}_{\omega_{2}}$ and $s\in\omega^{<\omega}$ such that the following holds:

\begin{enumerate}
\item $\{dom\left(  p_{\alpha}\right)  \mid\alpha\in X\}$ forms a $\Delta
$-system with root $R.$

\item $p_{\alpha}\upharpoonright R=\overline{p}$ for every $\alpha\in X.$

\item $\alpha\notin R$ for every $\alpha\in X.$

\item $p\left(  \alpha\right)  =s$ for every $\alpha\in X.$
\end{enumerate}

\qquad

It is clear that $\left\{  p_{\alpha}\mid\alpha\in\omega_{1}\right\}  $ is a
centered set (any finite set of conditions are compatible). Let $M$ be a
countable elementary submodel such that $\overline{p},\left\{  p_{\alpha}%
\mid\alpha\in\omega_{1}\right\}  ,$ $\dot{e},R\in M.$ Since $M\cap\omega_{2}$
is countable, we may find $\alpha,\beta\in X\setminus M$ such that $\alpha
\neq\beta$ and $dom\left(  p_{\alpha}\right)  \cap M=$ $dom\left(  p_{\beta
}\right)  \cap M=R.$ Let $m=U\left(  \delta_{\alpha},\delta_{\beta}\right)  .$
We may now find $\overline{q}$ and $i$ such that the following conditions hold:

\begin{enumerate}
\item $\overline{q}\in M$ and $\overline{q}\leq\overline{p}.$

\item $\overline{q}\Vdash``\dot{e}\left(  m\right)  =i\textquotedblright.$
\end{enumerate}

\qquad\ \ 

This is possible since $\overline{p},\dot{e}\in M$. Let $t\in\omega^{<\omega}$
be such that $\left\vert t\right\vert >i$ and $s\subseteq t.$ We now define a
condition $q$ as follows:\qquad\ \ 

\qquad\ \ \ \ \qquad\ \ \ \ \ \ \ \ \ 

$\hfill q\left(  \xi\right)  =\left\langle
\begin{array}
[c]{ccl}%
\overline{q}\left(  \xi\right)  &  & \text{if }\xi\in dom\left(  \overline
{q}\right) \\
p_{\alpha}\left(  \xi\right)  &  & \text{if }\xi\in dom\left(  p_{\alpha
}\right)  \setminus dom\left(  \overline{q}\right)  \text{ and }\xi\neq
\alpha\\
p_{\beta}\left(  \xi\right)  &  & \text{if }\xi\in dom\left(  p_{\beta
}\right)  \setminus dom\left(  \overline{q}\right)  \text{ and }\xi\neq\beta\\
t &  & \text{if }\xi=\alpha\text{ or }\xi=\beta
\end{array}
\right.  \hfill$

\qquad\ \ \ \qquad\ \ \ \qquad\ \ \ \qquad\ \ \bigskip

Note that this is possible since $dom\left(  \overline{q}\right)  \subseteq
M.$ Clearly $q\Vdash``\dot{F}\left(  \alpha,\beta\right)  >i\textquotedblright%
$ and $q\Vdash``\dot{e}(U(\dot{h}\left(  \alpha\right)  ,\dot{h}\left(
\beta\right)  ))=i\textquotedblright$ so $q\Vdash``\dot{F}\left(  \alpha
,\beta\right)  \neq\dot{e}(U(\dot{h}\left(  \alpha\right)  ,\dot{h}\left(
\beta\right)  ))\textquotedblright.$
\end{proof}

\qquad\ \ 

Since Cohen forcing has the countable chain condition, we conclude the following:

\begin{proposition}
There are no $\left(  1,\omega_{1}\right)  $-weakly universal graphs in the
Cohen model.\qquad
\end{proposition}

\subsection{The Sacks model}

The proof that there are no $\left(  1,\omega_{1}\right)  $-weakly universal
graph in the Side by Side Sacks model uses that the stick principle holds in
such model. It is then natural to wonder if the stick principle is enough to
get the non-existence of such graphs (under the failure of the Continuum
Hypothesis). Moreover, the stick principle already forbids the existence of
some universal graphs, as the following result of Shelah and Stepr\={a}ns shows:

\begin{proposition}
[\cite{UniversalGraphsandFunctionsonw1}]$_{\mid}^{\bullet}+\mathfrak{c}%
>\omega_{1}$ implies that there is no universal function $U:\left[  \omega
_{1}\right]  ^{2}\longrightarrow2.$
\end{proposition}

By the \emph{Sacks model} we mean a model obtained by forcing with
$\mathbb{S}_{\omega_{2}}$ over a model of the Generalized Continuum
Hypothesis. In this section, we will prove that there is a $\left(
1,\omega_{1}\right)  $-weakly universal graph in the Sacks model. The
following is a result of Mildenberger:

\begin{proposition}
[\cite{Heikeclub}]$\clubsuit$ holds in the Sacks model.
\end{proposition}

\qquad\ \ 

In particular, we will be able to conclude that the existence of a $\left(
1,\omega_{1}\right)  $-weakly universal graph is consistent with $\clubsuit.$
As usual, if $T\subseteq2^{<\omega}$ is a tree, we denote by $\left[
T\right]  $ the set of all branches (i.e. maximal linearly order sets) through
$T.$ Given $f\in2^{\omega}$ and $T\subseteq2^{<\omega}$ a finite tree, we say
that $f\in^{\ast}\left[  T\right]  $ if there is $n\in\omega$ such that
$f\upharpoonright n\in\left[  T\right]  .$ If $f\in^{\ast}\left[  T\right]  ,$
we define by $f\upharpoonright T$ to be the unique $t\in2^{<\omega}$ such that
there is $n$ for which $t=$ $f\upharpoonright n\in\left[  T\right]  .$ For
this section, we fix $W$ as the set of all $\left(  T,f\right)  $ such that
$T\subseteq2^{<\omega}$ is a finite tree and $f:\left[  T\right]
\longrightarrow\omega.$ It is easy to see that $W$ is a countable set.

\qquad\ \ \ 

We will need some definition and lemmas regarding iterated Sacks forcing. The
following is based on \cite{SacksMiller} and \cite{LifeintheSacksModel}. If
$p\in\mathbb{S}$ and $s\in2^{<\omega}$ we define $p_{s}=\left\{  t\in p\mid
t\subseteq s\vee s\subseteq t\right\}  .$ Note that $p_{s}$ is a Sacks tree if
and only if $s\in p.$ By $supp\left(  p\right)  $ we will denote the support
of $p.$

\begin{definition}
Let $p\in\mathbb{S}_{\alpha},$ $F\in\left[  supp\left(  p\right)  \right]
^{<\omega}$ and $\sigma:F\longrightarrow2^{n}.$ We define $p_{\sigma}$ as follows:

\begin{enumerate}
\item $supp\left(  p_{\sigma}\right)  =supp\left(  p\right)  .$

\item Letting $\beta<\alpha$ the following holds:

\begin{enumerate}
\item $p_{\sigma}\left(  \beta\right)  =p\left(  \beta\right)  $ if
$\beta\notin F.$

\item $p_{\sigma}\left(  \beta\right)  =p\left(  \beta\right)  _{\sigma\left(
\beta\right)  }$ if $\beta\in F.$
\end{enumerate}
\end{enumerate}
\end{definition}

\qquad\ \ 

Similar to previous situation, $p_{\sigma}$ is not necessarily a condition of
$\mathbb{S}_{\alpha}.$ We will say that $\sigma:F\longrightarrow2^{n}$ is
\emph{consistent with }$p$ if $p_{\sigma}\in\mathbb{S}_{\alpha}.$ A condition
$p$ is $\left(  F,n\right)  $\emph{-determined }if for every $\sigma
:F\longrightarrow2^{n}$ either $\sigma$ is consistent with $p$ or there is
$\beta\in F$ such that $\sigma\upharpoonright\left(  F\cap\beta\right)  $ is
consistent with $p$ and $(p\upharpoonright\beta)_{\sigma\upharpoonright\left(
F\cap\beta\right)  }\Vdash``\sigma\left(  \beta\right)  \notin p\left(
\beta\right)  \textquotedblright.$

\qquad\ \ 

We say that $p\in\mathbb{S}_{\alpha}$ is \emph{determined }if for every
$F\in\left[  supp\left(  p\right)  \right]  ^{<\omega}$ and for every
$n\in\omega$ there are $G$ and $m$ such that the following holds:

\begin{enumerate}
\item $G\in\left[  supp\left(  p\right)  \right]  ^{<\omega}.$

\item $F\subseteq G.$

\item $n<m.$

\item $p$ is $\left(  G,m\right)  $-determined.
\end{enumerate}

\qquad\ \ \ \qquad\ \qquad\ \ \ 

The following result is well known:

\begin{lemma}
[\cite{SacksMiller}]For every $p\in\mathbb{S}_{\alpha}$ there is a determined
$q\leq p$.
\end{lemma}

\qquad\ \ \ \qquad\ \ \ \qquad\qquad\ \ \ 

Let $p$ be a determined condition. We say that $\left\langle \left(
F_{i},n_{i},\Sigma_{i}\right)  \mid i\in\omega\right\rangle $ is \emph{a
representation of }$p$ if the following holds:

\begin{enumerate}
\item $F_{i}\in\left[  supp\left(  p\right)  \right]  ^{<\omega},$ $n_{i}%
\in\omega.$

\item $F_{i}\subseteq F_{i+1}$ and $n_{i}<n_{i+1}.$

\item $supp\left(  p\right)  =%
{\textstyle\bigcup\limits_{i\in\omega}}
F_{i}.$

\item $p$ is $\left(  F_{i},n_{i}\right)  $-determined for every $i\in\omega.$

\item $\Sigma_{i}$ is the set of all $\sigma:F_{i}\longrightarrow2^{n_{i}}$
such that $\sigma$ is consistent with $p.$
\end{enumerate}

\qquad\ \ \ 

We will also need the following definition:

\begin{definition}
Let $p\in\mathbb{S}_{\alpha}$ be a determined condition and $\dot{r}$ an
$\mathbb{S}_{\alpha}$-name for an element of $2^{\omega}.$ We say that $p$ is
$\dot{r}$\emph{-canonical }if there are $\left\langle \left(  F_{i}%
,n_{i},\Sigma_{i}\right)  \mid i\in\omega\right\rangle $ and $\left\langle
\mathcal{C}_{i}\mid i\in\omega\right\rangle $ with the following properties:

\begin{enumerate}
\item $\left\{  \left(  F_{i},n_{i},\Sigma_{i}\right)  \mid i\in
\omega\right\}  $ is a representation of $p.$

\item $\mathcal{C}_{i}=\left\{  C_{\sigma}\mid\sigma\in\Sigma_{i}\right\}  $
is a collection of disjoint clopen subsets of $2^{\omega}.$

\item For every $\sigma\in\Sigma_{i}$ there is $s_{\sigma}\in2^{n_{i}}$ such
that $C_{\sigma}\subseteq\left\langle s_{\sigma}\right\rangle .$\footnote{If
$t\in2^{<\omega}$ we define $\left\langle t\right\rangle =\left\{
x\in2^{\omega}\mid t\subseteq x\right\}  .$}

\item If $i\in\omega$ and $\sigma\in\Sigma_{i},$ then $p_{\sigma}\Vdash
``\dot{r}\in C_{\sigma}\textquotedblright$ (in particular, $p_{\sigma}$
determines $\dot{r}\upharpoonright n_{i}$).
\end{enumerate}
\end{definition}

\qquad\ \ \ \ \ \ \ \qquad\ \ \ \qquad\ \ 

In the above situation, we say that $\left\langle \left(  F_{i},n_{i}%
,\Sigma_{i},\mathcal{C}_{i}\right)  \mid i\in\omega\right\rangle $ is an
$\dot{r}$\emph{-canonical representation for} $p.$ The following is lemma 6 of
\cite{SacksMiller}:

\begin{lemma}
[\cite{SacksMiller}]Let $\alpha\leq\omega_{2}$ $,p\in\mathbb{S}_{\alpha}$ and
$\dot{r}$ an $\mathbb{S}_{\alpha}$-name for an element of $2^{\omega}$ such
that $p\Vdash``\dot{r}\notin%
{\textstyle\bigcup\limits_{\beta<\alpha}}
V\left[  G_{\beta}\right]  \textquotedblright.$ There is $q\leq p$ such that
$q$ is $\dot{r}$-canonical.
\end{lemma}

\qquad\ \ \ \ \qquad\ \ \ \qquad\ \ \ 

With the same proof of the previous lemma, it is possible to prove the following:

\begin{lemma}
Let $\alpha\leq\omega_{2}$ $,p\in\mathbb{S}_{\alpha}$, $\dot{r}$ an
$\mathbb{S}_{\alpha}$-name for an element of $2^{\omega}$ such that
$p\Vdash``\dot{r}\notin%
{\textstyle\bigcup\limits_{\beta<\alpha}}
V\left[  G_{\beta}\right]  \textquotedblright$ and $\dot{g}$ an $\mathbb{S}%
_{\alpha}$-name for an element of $\omega^{\omega}.$ There is $q\leq p$ such
that $q$ is $\dot{r}$-canonical with $\dot{r}$-canonical representation
$\left\langle \left(  F_{i},n_{i},\Sigma_{i},\mathcal{C}_{i}\right)  \mid
i\in\omega\right\rangle $ and there is $\left\langle h_{i}\mid i\in
\omega\right\rangle $ such that the following conditions:

\begin{enumerate}
\item $h_{i}:\Sigma_{i}\longrightarrow\omega$ for every $i\in\omega.$

\item If $i\in\omega$ and $\sigma\in\Sigma_{i}$ then $q_{\sigma}\Vdash
``\dot{g}\left(  i\right)  =h_{i}\left(  \sigma\right)  \textquotedblright.$
\end{enumerate}
\end{lemma}

\qquad\ \ \qquad\ \ 

The lemma 6 of \cite{SacksMiller} is proved using a fusion argument. To prove
the previous lemma we use the same fusion argument, with the extra step of
deciding the respective value of $\dot{g}$ at each step. We leave the details
for the reader. As before, in the above situation we say that $\left\langle
\left(  F_{i},n_{i},\Sigma_{i},\mathcal{C}_{i},h_{i}\right)  \mid i\in
\omega\right\rangle $ is an $\left(  \dot{r},\dot{g}\right)  $\emph{-canonical
representation for} $q.$

We can now prove the following:

\begin{proposition}
Let $\eta<\omega_{2},\dot{g}$ and $p\in\mathbb{S}_{\eta+1}$ such that
$p\Vdash``\dot{g}:\omega\longrightarrow\omega\textquotedblright.$ There is a
determined $q\in\mathbb{S}_{\eta+1}$ and $\left\{  \left(  n,T_{n}%
,f_{n}\right)  \mid n\in\omega\right\}  $ with $\left\{  \left(  T_{n}%
,f_{n}\right)  \mid n\in\omega\right\}  \subseteq W$ such that the following holds:

\begin{enumerate}
\item $q\leq p.$

\item $q\Vdash``\dot{r}_{\eta}\in^{\ast}\left[  T_{n}\right]
\textquotedblright$ for each $n\in\omega.$

\item $q\Vdash``\dot{g}\left(  n\right)  =f_{n}\left(  \dot{r}_{\eta
}\upharpoonright T_{n}\right)  \textquotedblright$ for every $n\in\omega.$
\end{enumerate}
\end{proposition}

\begin{proof}
By the previous lemma, we can find $p_{1}\leq p$ that has an $\left(  \dot
{r}_{\eta},\dot{g}\right)  $-canonical representation $\left\{  \left(
F_{i},n_{i},\Sigma_{i},\mathcal{C}_{i},h_{i}\right)  \mid i\in\omega\right\}
$. We now have the following interesting property: If $G\subseteq
\mathbb{S}_{\eta+1}$ is a generic filter with $p_{1}\in G$ and $\left\langle
r_{\alpha}\right\rangle _{\alpha\leq\eta}$ is the generic sequence, then the
following holds in $V\left[  G\right]  :$

\begin{description}
\item[*)] For every $i\in\omega$ and $\sigma\in\Sigma_{i}$, if $r_{\eta}\in
C_{\sigma}$ then $\sigma\left(  \beta\right)  \subseteq r_{\beta}$ for every
$\beta\in F_{i}.$
\end{description}

\qquad\ \ \ 

This property holds because $\mathcal{C}_{i}=\left\{  C_{\sigma}\mid\sigma
\in\Sigma_{i}\right\}  $ is a collection of disjoint sets. In this way,
$r_{\eta}$ is able to \textquotedblleft code\textquotedblright\ each of the
previews generic reals. Let $Y$ be the set of all maximal $z\in2^{<\omega}$
with the property that $\left\langle z\right\rangle \subseteq%
{\textstyle\bigcup}
\mathcal{C}_{i}.$ Note that since $\mathcal{C}_{i}$ is a finite set of clopen
sets, $Y$ is a finite set. Let $T_{i}$ be the smallest finite tree such that
$Y\subseteq T_{i}.$ Note that $T_{i}$ has the following properties:

\begin{enumerate}
\item $%
{\textstyle\bigcup\limits_{s\in\left[  T_{i}\right]  }}
\left\langle s\right\rangle =%
{\textstyle\bigcup}
\mathcal{C}_{i}.$

\item For every $s\in\left[  T_{i}\right]  $ there is exactly one $\sigma
\in\Sigma_{i}$ for which $\left\langle s\right\rangle \subseteq C_{\sigma}$
(where $\mathcal{C}_{i}=\left\{  C_{\sigma}\mid\sigma\in\Sigma_{i}\right\}  $).
\end{enumerate}

\qquad\ \ \ 

For every $i$ we have the following properties:

\begin{enumerate}
\item $p_{1}\Vdash``\dot{r}_{\eta}\in\left[  T_{i}\right]  \textquotedblright%
.$

\item Let $G\subseteq\mathbb{S}_{\eta+1}$ be a generic filter with $p_{1}\in
G$ and $\sigma\in\Sigma_{i}.$ If $r_{\eta}\in C_{\sigma}$ then $\left(
p_{1}\right)  _{\sigma}\in G.$
\end{enumerate}

\qquad\ \ \ 

We now have the following claim:

\begin{claim}
If $i\in\omega$, $s\in\left[  T_{i}\right]  $ and $q_{0},q_{1}\ $are two
conditions extending$\ p_{1}$ such that $q_{i}\Vdash``s\subseteq\dot{r}_{\eta
}\textquotedblright$ for $j\in\left\{  0,1\right\}  $ then there is
$k\in\omega$ such that $q_{0}\Vdash``\dot{g}\left(  i\right)
=k\textquotedblright$ and $q_{1}\Vdash``\dot{g}\left(  i\right)
=k\textquotedblright.$
\end{claim}

\qquad\ \ \ 

We will prove the claim. Let $\sigma\in\Sigma_{i}$ such that $\left\langle
s\right\rangle \subseteq C_{\sigma}$ and let $j<2.$ Note that since $q_{j}\leq
p_{1}$ and $q_{j}\Vdash``\dot{r}_{\eta}\in C_{\sigma}\textquotedblright,$ it
follows that $q_{j}\Vdash``\left(  p_{1}\right)  _{\sigma}\in\dot
{G}\textquotedblright$ (where $\dot{G}$ is a name for the generic filter),
hence $q_{j}\Vdash``\dot{g}\left(  i\right)  =h_{i}\left(  \sigma\right)
\textquotedblright,$ the claim follows.

\qquad\ \ 

For every $n\in\omega,$ we define a function $f_{n}:\left[  T_{n}\right]
\longrightarrow\omega$ as follows: for every $s\in\left[  T_{n}\right]  ,$ let
$f_{n}\left(  s\right)  $ such that for every $q\leq p_{1}$ if $q\Vdash
``s\subseteq\dot{r}_{\eta}\textquotedblright$ then $q\Vdash``\dot{g}\left(
n\right)  =f_{n}\left(  s\right)  \textquotedblright.$ Note that $f_{n}$ is
well defined by the previous claim. It is easy to see that $\left\{  \left(
n,T_{n},f_{n}\right)  \mid n\in\omega\right\}  $ has the desired properties.
\end{proof}

\qquad\ \ \ \qquad\ \qquad\ \ \ \ 

We will say that a graph $U:\left[  \omega_{1}\right]  ^{2}\longrightarrow W$
is $\left(  1,\omega_{1}\right)  $\emph{-weakly universal} if for every
$F:\left[  \omega_{1}\right]  ^{2}\longrightarrow\omega$ there is an injective
$h:\omega_{1}\longrightarrow\omega_{1}$ and a function $e:W\longrightarrow
\omega$ such that $F\left(  \alpha,\beta\right)  =e\left(  U\left(  h\left(
\alpha\right)  ,h\left(  \beta\right)  \right)  \right)  .$ As expected, we
have the following result:

\begin{lemma}
If there is a $U:\left[  \omega_{1}\right]  ^{2}\longrightarrow W$ which is
$\left(  1,\omega_{1}\right)  $-weakly universal, then there is $U_{1}:\left[
\omega_{1}\right]  ^{2}\longrightarrow\omega$ that is $\left(  1,\omega
_{1}\right)  $-weakly universal.
\end{lemma}

\begin{proof}
Let $U:\left[  \omega_{1}\right]  ^{2}\longrightarrow W$ be a $\left(
1,\omega_{1}\right)  $-weakly universal graph. Fix $g:W\longrightarrow\omega$
a bijective function. We define $U_{1}:\left[  \omega_{1}\right]
^{2}\longrightarrow\omega$ where $U_{1}\left(  \alpha,\beta\right)  =g\left(
U\left(  \alpha,\beta\right)  \right)  .$ It is easy to see that $U_{1}$ is
$\left(  1,\omega_{1}\right)  $-weakly universal.
\end{proof}

\qquad\ \ \ 

For the rest of this section, we will assume the Continuum Hypothesis. Fix a
large enough regular cardinal $\theta>\left(  2^{\omega_{2}}\right)  ^{+}.$ We
will now fix $\overline{M}=\{(M_{\alpha},\in,\Vdash_{\mathbb{S}_{\eta_{\alpha
}+1}},p_{\alpha},\eta_{\alpha},\xi_{\alpha},\dot{g}_{\alpha})\mid\alpha
\in\omega_{1}\}$ with the following properties:

\begin{enumerate}
\item $M_{\alpha}$ is a countable elementary submodel of $H\left(
\theta\right)  $ such that $p_{\alpha},\eta_{\alpha},\xi_{\alpha},\dot
{g}_{\alpha}\in M_{\alpha}.$

\item $\eta_{\alpha}<\omega_{2}$ and $p_{\alpha}\in\mathbb{S}_{\eta_{\alpha
}+1}.$

\item $\xi_{\alpha}<\omega_{1}$ and $p_{\alpha}\Vdash``\dot{g}_{\alpha}%
:\xi_{\alpha}\longrightarrow\omega\textquotedblright.$

\item For every $(N,\in,\Vdash_{\mathbb{S}_{\eta+1}},p,\eta,\xi,\dot{g})$ if
the following properties hold:

\begin{enumerate}
\item $N$ is a countable elementary submodel of $H\left(  \theta\right)  $
such that $p,\eta,\xi,\dot{g}\in N.$

\item $\eta<\omega_{2}$ and $p\in\mathbb{S}_{\eta_{+1}}.$

\item $\xi<\omega_{1}$ and $p\Vdash``\dot{g}:\xi\longrightarrow\omega
\textquotedblright.$

Then, there is $\alpha<\omega_{1}$ such that $(M_{\alpha},\in,\Vdash
_{\mathbb{S}_{\eta_{\alpha}+1}},p_{\alpha},\eta_{\alpha},\xi_{\alpha},\dot
{g}_{\alpha})$ and $(N,\in,\Vdash_{\mathbb{S}_{\eta+1}},p,\eta,\xi,\dot{g})$
are isomorphic.
\end{enumerate}
\end{enumerate}

\qquad\ \ \ 

This is possible since $\mathbb{S}_{\eta+1}$ is proper and we are assuming the
Continuum Hypothesis. For every $\alpha<\omega_{1},$ let $\delta_{\alpha
}=M_{\alpha}\cap\omega_{1}.$ We now choose $\left\{  \beta_{\alpha}\mid
\alpha\in\omega_{1}\right\}  \subseteq\omega_{1}$ such that $\delta_{\alpha
}<\beta_{\alpha}$ and if $\alpha_{1}\neq\alpha_{2}$ then $\beta_{\alpha_{1}%
}\neq\beta_{\alpha_{2}}.$ For every $\alpha<\omega_{1},$ we also fix an
enumeration $\xi_{\alpha}=\left\{  \xi_{\alpha}\left(  n\right)  \mid
n\in\omega\right\}  .$ By the previous lemmas, for every $\alpha<\omega_{1},$
we can find $q_{\alpha},$ $\left\{  \left(  n,T_{n}^{\alpha},f_{n}^{\alpha
}\right)  \mid n\in\omega\right\}  $ such that the following holds:

\begin{enumerate}
\item $q_{\alpha}\in\mathbb{S}_{\eta_{\alpha}+1}\cap M_{\alpha}$ and
$q_{\alpha}\leq p_{\alpha}.$

\item $\left\{  \left(  T_{n}^{\alpha},f_{n}^{\alpha}\right)  \mid n\in
\omega\right\}  \subseteq W.$

\item $q_{\alpha}\Vdash``\dot{r}_{\eta_{\alpha}}\in^{\ast}\left[
T_{n}^{\alpha}\right]  \textquotedblright$ for each $n\in\omega.$

\item $q_{\alpha}\Vdash``\dot{g}_{\alpha}\left(  \xi_{\alpha}\left(  n\right)
\right)  =f_{n}^{\alpha}\left(  \dot{r}_{\eta_{\alpha}}\upharpoonright
T_{n}^{\alpha}\right)  \textquotedblright$ for every $n\in\omega.$
\end{enumerate}

\qquad\ \ \qquad\ 

We now define the graph $U:\left[  \omega_{1}\right]  ^{2}\longrightarrow W$
as follows: given $\alpha<\omega_{1}$ and $n\in\omega$ we define $U\left(
\xi_{\alpha}\left(  n\right)  ,\beta_{\alpha}\right)  =\left(  T_{n}^{\alpha
},f_{n}^{\alpha}\right)  $ (the value of $U$ is not important in any other
case, so if a pair $\left(  \nu_{1},\nu_{2}\right)  $ is not of the form
$\left(  \xi_{\alpha}\left(  n\right)  ,\beta_{\alpha}\right)  ,$ we can let
$U\left(  \nu_{1},\nu_{2}\right)  $ be any element of $W$ otherwise). We will
show that $U$ is forced to be $\left(  1,\omega_{1}\right)  $-weakly
universal. Given $\eta<\omega_{2}$, in the forcing extension, we define the
function $e_{\eta}:W\longrightarrow\omega$ given by $e_{\eta}\left(
T,f\right)  =f\left(  \dot{r}_{\eta}\upharpoonright T\right)  $ if $\dot
{r}_{\eta}\in^{\ast}\left[  T\right]  $ and $e_{\eta}\left(  T,f\right)  =0$
in other case. We need the following lemma:

\begin{lemma}
Let $G\subseteq\mathbb{S}_{\omega_{2}}$ be a generic filter. Let $\eta
<\omega_{2},\xi<\omega_{1}$ and $g:\xi\longrightarrow\omega$ such that $g\in
V\left[  G_{\eta+1}\right]  .$ \ There is $\alpha\in\omega_{1}$ such that the
following holds:

\begin{enumerate}
\item $\xi_{\alpha}=\xi.$

\item $g\left(  \xi_{\alpha}\left(  n\right)  \right)  =e_{\eta}(U\left(
\xi_{\alpha}\left(  n\right)  ,\beta_{\alpha}\right)  )$ for every $n\in
\omega.$
\end{enumerate}
\end{lemma}

\begin{proof}
It is enough to show that the conditions that force the above properties are
dense, in this way, there will be such condition in the generic filter. Let
$p\in\mathbb{S}_{\eta+1},$ we will see we can extend $p$ to get the desired
conclusion. Let $N$ be a countable elementary submodel such that $p,\eta
,\xi,\dot{g}\in N.$ We first find $\alpha<\omega_{1}$ such that $(M_{\alpha
},\in,\Vdash_{\mathbb{S}_{\eta_{\alpha}+1}},p_{\alpha},\eta_{\alpha}%
,\xi_{\alpha},\dot{g}_{\alpha})$ and $(N,\in,\Vdash_{\mathbb{S}_{\eta+1}%
},p,\eta,\xi,\dot{g})$ are isomorphic. Let $\pi:M_{\alpha}\longrightarrow N$
be the isomorphism and let $q=$ $\pi\left(  q_{\alpha}\right)  .$ Note that
the isomorphism fixes every ordinal smaller than $\delta_{\alpha}$ (in
particular each $\xi_{\alpha}\left(  n\right)  $) as well as each element in
$W.$ By the isomorphism, the following conditions hold:

\begin{enumerate}
\item $q\in\mathbb{S}_{\eta_{+1}}\cap N$ and $q\leq p.$

\item $q\Vdash``\dot{r}_{\eta}\in^{\ast}\left[  T_{n}^{\alpha}\right]
\textquotedblright$ for each $n\in\omega.$

\item $q\Vdash``\dot{g}\left(  \xi_{\alpha}\left(  n\right)  \right)
=f_{n}^{\alpha}\left(  \dot{r}_{\eta}\upharpoonright T_{n}^{\alpha}\right)
\textquotedblright$ for every $n\in\omega.$
\end{enumerate}

\qquad\ \ 

By the last clause, it follows that $q\Vdash``\dot{g}\left(  \xi_{\alpha
}\left(  n\right)  \right)  =e_{\eta}(U\left(  \xi_{\alpha}\left(  n\right)
,\beta_{\alpha}\right)  )\textquotedblright.$
\end{proof}

\qquad\ \ \ 

We can then prove the following:

\begin{proposition}
There is a $\left(  1,\omega_{1}\right)  $-weakly universal graph in the Sacks model.
\end{proposition}

\begin{proof}
We will show that $U$ is forced to be a $\left(  1,\omega_{1}\right)  $-weakly
universal graph (note that this is enough by lemma 17). Let $p\in
\mathbb{S}_{\omega_{2}}$ and $\dot{F}$ such that $p\Vdash``\dot{F}:\left[
\omega_{1}\right]  ^{2}\longrightarrow\omega\textquotedblright.$ Since Sacks
forcing has the $\omega_{2}$-chain condition, we may assume that there is
$\eta<\omega_{2}$ such that $p\in\mathbb{S}_{\eta}$ and $\dot{F}$ is an
$\mathbb{S}_{\eta}$-name.

\qquad\ \ 

Given $\gamma\leq\omega_{1},$ we will say that an injective function
$h:\gamma\longrightarrow\omega_{1}$ is a \emph{partial }$e_{\eta}%
$\emph{-embedding }if $F\left(  \alpha,\beta\right)  =e_{\eta}(U\left(
h\left(  \alpha\right)  ,h\left(  \beta\right)  \right)  )$ for every
$\alpha,\beta<\gamma.$ Let $G$ be a generic filter such that $p\in G.$ We
claim that in $V\left[  G_{\eta+1}\right]  $ the following holds:

\begin{description}
\item[*)] If $h:\gamma\longrightarrow\omega_{1}$ is a partial $e_{\eta}%
$-embedding with $\gamma<\omega_{1},$ then there is a partial $e_{\eta}%
$-embedding $\overline{h}:\gamma+1\longrightarrow\omega_{1}$ extending $h.$
\end{description}

\qquad\ \ 

We argue in $V\left[  G_{\eta+1}\right]  .$ Let $\xi=%
{\textstyle\bigcup}
h\left[  \gamma\right]  +1$ and note we can find $g\in V\left[  G_{\eta
+1}\right]  $ such that $g:\xi\longrightarrow\omega$ and $g\left(  h\left(
\delta\right)  \right)  =F\left(  \delta,\gamma\right)  $ for all
$\delta<\gamma.$ By the previous lemma, there is $\alpha\in\omega_{1}$ such
that $\xi_{\alpha}=\xi$ and $g\left(  \xi_{\alpha}\left(  n\right)  \right)
=e_{\eta}(U\left(  \xi_{\alpha}\left(  n\right)  ,\beta_{\alpha}\right)  ).$
We now define $\overline{h}=h\cup\left\{  \left(  \gamma,\beta_{\alpha
}\right)  \right\}  .$ Note that $\beta_{\alpha}\notin h\left[  \gamma\right]
$ since $h\left[  \gamma\right]  \subseteq\xi=\xi_{\alpha}<\delta_{\alpha
}<\beta_{\alpha}.$ We only need to prove that $\overline{h}$ is a partial
$e_{\eta}$-embedding. Let $\delta<\gamma,$ we can find $n\in\omega$ such that
$h\left(  \delta\right)  =\xi_{\alpha}\left(  n\right)  .$ It then follows that:

\qquad\ \ %

\begin{tabular}
[t]{lll}%
$e_{\eta}(U\left(  \overline{h}\left(  \delta\right)  ,\overline{h}\left(
\gamma\right)  \right)  )$ & $=$ & $e_{\eta}(U\left(  \xi_{\alpha}\left(
n\right)  ,\beta_{\alpha}\right)  )$\\
& $=$ & $g\left(  \xi_{\alpha}\left(  n\right)  \right)  $\\
& $=$ & $g\left(  h\left(  \delta\right)  \right)  $\\
& $=$ & $F\left(  \delta,\gamma\right)  $%
\end{tabular}

\qquad\ \ \bigskip

This finishes the claim. It is clear that any maximal $e_{\eta}$-embedding
will embed $F$ into $U.$
\end{proof}

\subsection{Open questions}

In general, a function $U:\left[  \omega_{1}\right]  ^{2}\longrightarrow
\omega$ is $\left(  1,\kappa\right)  $-\emph{weakly universal }if for every
$F:\left[  \omega_{1}\right]  ^{2}\longrightarrow\omega$ there is an injective
function $h:\omega_{1}\longrightarrow\omega_{1}$ and a function $e:\omega
\longrightarrow\omega$ such that $\left\vert e^{-1}\left(  n\right)
\right\vert <\kappa$ for every $n\in\omega$ and \ $F\left(  \alpha
,\beta\right)  =e\left(  U\left(  h\left(  \alpha\right)  ,h\left(
\beta\right)  \right)  \right)  $ for every $\alpha,\beta\in\omega_{1}$. It
would be interesting to know the answer of the following question:

\begin{problem}
Are there $\left(  1,\omega\right)  $-weakly universal functions (or even
$\left(  1,2\right)  $-weakly universal functions) in the Sacks model?
\end{problem}

\qquad\ \ \ \ \qquad\ \ \ \ \ \qquad\ \ 

In fact, we conjecture that $_{\mid}^{\bullet}+\mathfrak{c}>\omega_{1}$
implies that there are no $\left(  1,\omega\right)  $-weakly universal functions.

\qquad\ \ \ 

We would like to mention that there are no $\left(  1,\omega_{1}\right)
$-weakly universal functions after performing a pseudo-iteration of Cohen
forcing, as described in \cite{SticksandClubs}. It would be interesting to
know what kind of universal graphs exist on the \textquotedblleft canonical
models\textquotedblright\ of set theory.\qquad\ \ \ 

\begin{problem}
Are there $\left(  1,\omega_{1}\right)  $-weakly universal functions in the
random, Hechler, Laver, Miller and Mathias models?
\end{problem}

\qquad\ \ \ \qquad\ \ \ \ 

The purpose of the \textsf{CPA }axioms introduced in \cite{CPAbook} is to
provide an axiomatization of the Sacks model. In light of this work, it is
then natural to ask the following:

\begin{problem}
Does the existence of $\left(  1,\omega_{1}\right)  $-weakly universal
functions follow from one of the \textsf{CPA }axioms?
\end{problem}

\begin{acknowledgement}
The author would like to thank Juris Stepr\={a}ns for several discussions on
this topic and hours of stimulating conversations. I would also like to thank
Damjan Kalajdzievski for several comments and discussions regarding universal
functions. The author would also like to thank the generous referee for all
his/her careful corrections and comments.
\end{acknowledgement}

\bibliographystyle{plain}
\bibliography{universal}

\qquad\ \ 

Osvaldo Guzm\'{a}n

York University

oguzman9@yorku.ca
\end{document}